\documentclass[11pt]{article}
\usepackage[utf8]{inputenc}
\usepackage{enumerate}
\usepackage{bbm}
\usepackage[a4paper]{geometry}
\usepackage{bigints}
\usepackage{amsthm}
\usepackage{amssymb}
\usepackage{color}
\usepackage{tikz}
\newtheorem{theorem}{Theorem}
\newtheorem{proposition}{Proposition}
\newtheorem{corollary}{Corollary}
\newtheorem{lemma}{Lemma}

\theoremstyle{definition}
\newtheorem{definition}{Definition}

\theoremstyle{remark}
\newtheorem{remark}{Remark}

\def\E{\mathbb{E}}
\def\P{\mathbb{P}}
\def\R{\mathbb{R}}
\def\Q{\mathbb{Q}}
\def\G{\mathbb{G}}

\def\1{\mathbbm{1}}

\def\cF{{\cal F}}

\def\cD{\mathcal{D}}
\def\cC{\mathcal{C}}
\def\cM{\mathcal{M}}
\def\fB{\mathfrak{B}}

\title{Quasi-stationarity for one-dimensional renormalized Brownian motion}
\author{William Oçafrain$^{1}$}
\date{\today}

\begin{document}

\footnotetext[1]{Institut de Math\'{e}matiques de Toulouse, UMR 5219; Universit\'{e} de Toulouse, CNRS, UPS IMT, F-31062
Toulouse Cedex 9, France; \\
  E-mail: william.ocafrain@math.univ-toulouse.fr}

\maketitle

\begin{abstract}
We are interested in the quasi-stationarity for the time-inhomogeneous Markov process
$$X_t = \frac{B_t}{(t+1)^\kappa}$$
where $(B_t)_{t \geq 0}$ is a one-dimensional Brownian motion and $\kappa \in (0,\infty)$. We first show that the law of $X_t$ conditioned not to go out from $(-1,1)$ until time $t$ converges weakly towards the Dirac measure $\delta_0$ when $\kappa > \frac{1}{2}$, when $t$ goes to infinity. Then, we show that this conditional probability measure converges weakly towards the quasi-stationary distribution for an Ornstein-Uhlenbeck process when $\kappa = \frac{1}{2}$. Finally, when $\kappa < \frac{1}{2}$, it is shown that the conditional probability measure converges towards the quasi-stationary distribution for a Brownian motion. We also prove the existence of a $Q$-process and a quasi-ergodic distribution for $\kappa = \frac{1}{2}$ and $\kappa < \frac{1}{2}$.    
\end{abstract} 

\bigskip

\bigskip

\maketitle

\section{Introduction}

In this paper, we are interested in some notions related to quasi-stationarity for a one-dimensional Brownian motion $(B_t)_{t \geq 0}$ killed when reaching the moving boundary $t \to \{-(t+1)^\kappa,(t+1)^\kappa\}$, with $\kappa \in (0,\infty)$. Quasi-stationarity with moving boundaries was studied in \cite{ocafrain2018} and \cite{OCAFRAIN2017} for periodic or converging boundaries, but expanding boundaries were not yet considered. Actually, instead of considering the process $B$ absorbed at $t \to \{-(t+1)^\kappa,(t+1)^\kappa\}$, we will study the quasi-stationarity for the process $X = (X_t)_{t \geq 0}$ absorbed at $(-1,1)^c$ and defined by
$$X_t := \frac{B_t}{(t+1)^\kappa}, ~~~~\forall t < \tau_X,$$
where $\tau_X := \inf\{t \geq 0 : |X_t| = 1\}$.

The process $X$ is a time-inhomogeneous Markov process. For any $x \in \R$ and $s \geq 0$, denote by $\P_{x,s}$ the probability measure satisfying $\P_{x,s}(X_s = x) = 1$, and denote by $\E_{x,s}$ its corresponding expectation. Also, for any measure $\mu$, for any $s \geq 0$, one denotes by $\P_{\mu,s} := \int_{\R} \P_{x,s} \mu(dx)$ and $\E_{\mu,s} := \int_{\R} \E_{x,s} \mu(dx)$. 

A \textit{quasi-stationary distribution for $X$ absorbed at $(-1,1)^c$} is a probability measure $\alpha$ supported on $(-1,1)$ such that 
$$\P_{\alpha,s}(X_t \in \cdot | \tau_X > t) = \alpha,~~~~\forall s \leq t.$$
We refer the reader to \cite{CMSM,MV2012} for more details on the theory. Note however that these references only deal with the time-homogeneous setting and that quasi-stationary distributions for time-inhomogeneous Markov processes do not exist except for particular cases (especially we will see that the existence of one quasi-stationary distribution holds only for $\kappa = \frac{1}{2}$). 

Usually, in the literature dealing with quasi-stationarity, mathematicians are interested in showing the weak convergence of marginal laws of Markov processes conditioned not to be absorbed by a cemetery set. The corresponding limiting probability measure is called \textit{quasi-limiting distribution}. For our purpose, we define a quasi-limiting distribution as follows:
\begin{definition}
\label{qld}
 $\alpha$ is a quasi-limiting distribution for $X$ if, for some initial law $\mu$ supported on $(-1,1)$ and for any $s \geq 0$,
$$\lim_{t \to \infty} \P_{\mu,s}(X_t \in \cdot | \tau_X > t) = \alpha,$$
where the limit refers to the weak convergence of measures.
\end{definition}   
In \cite{MV2012} it is noted that, in the time-homogeneous setting, quasi-stationary distribution and quasi-limiting distribution are equivalent notions. In the time-inhomogeneous setting, this equivalence does not hold anymore. More particularly a time-inhomogeneous Markov process could admit a quasi-limiting distribution without admitting a quasi-stationary distribution. However, a quasi-stationary distribution is necessarily a quasi-limiting distribution. 

Quasi-limiting distribution is not the only point of interest in the theory of quasi-stationarity. Another point is the $Q$-\textit{process}, which can be considered as the law of the considered Markov process conditioned \textit{"never to be absorbed"}. Concerning the process $X$, we define the $Q$-process as follows:
\begin{definition}
\label{qproc-def}
We say that there is a $Q$-process for $X$ if there exists a family $(\Q_{x,s})_{x \in (-1,1), s \geq 0}$ of probability measures defined by : for any $x \in (-1,1)$ and for any $s \leq t$,
$$\Q_{x,s}(X_{[s,t]} \in \cdot) := \lim_{T \to \infty} \P_{x,s}(X_{[s,t]} \in \cdot | T < \tau_X),$$
where, for any $u \leq v$, $X_{[u,v]}$ is the trajectory of $X$ between times $u$ and $v$. Then the $Q$-process is defined as the law of $X$ under $(\Q_{x,s})_{x \in (-1,1), s \geq 0}$.
\end{definition}   
In general, the $Q$-process is also a Markov process and the theory of quasi-stationarity allows to get some results of ergodicity for the $Q$-process. In particular, under some assumptions (see for example \cite{CV2014,CV2017c,V2018}), the Q-process admits a stationary
distribution which is absolutely continuous with respect to the quasi-stationary distribution.

Finally, a third concept to study is the \textit{quasi-ergodic distribution}, defined as follows:
\begin{definition}
\label{qed}
$\beta$ is a quasi-ergodic distribution for $X$ if, for some initial law $\mu$ supported on $(-1,1)$ and for any $s \geq 0$,
$$\lim_{t \to \infty} \frac{1}{t} \int_s^t \P_{\mu,s}(X_u \in \cdot | \tau_X > t)du = \beta.$$
\end{definition}
In the literature, this notion is also called \textit{mean-ratio quasi-stationary distribution}. The references \cite{CMSM,MV2012} do not deal with quasi-ergodic distributions. See for example \cite{BR1999, CV2017} which provide general assumptions implying the existence of quasi-ergodic distributions for time-homogeneous Markov processes. In particular, it is shown in \cite{BR1999} that, if the Q-process is Harris recurrent, the quasi-ergodic distribution is the stationary distribution of the Q-process. Concerning the time-inhomogeneous setting, similar results can be stated (see \cite{ocafrain2018}) when the Q-process converges
weakly at infinity. In this case, the quasi-ergodic distribution coincides with the limiting probability measure.   
    
Some general results on quasi-stationarity for time-inhomogeneous Markov process are established, particularly in \cite{CV2016}, where it is shown that criteria based on Doeblin-type condition implies a mixing property (or \textit{merging} or \textit{weak ergodicity}) and the existence of the $Q$-process. However, for our purpose, these conditions will be difficult to establish. See also  \cite{villemonais2013,DMV2018,OCAFRAIN2017,ocafrain2018} for a few results about quasi-stationarity in the time-inhomogeneous setting, and \cite{BCG2017} for ergodic properties for general non-conservative (time-homogeneous and inhomogeneous) semi-group.

Now let us come back to our process $X$. As we can expect, the existence of a quasi-limiting, $Q$-process and quasi-ergodic distribution will strongly depend on $\kappa$. More precisely, three regimes are identified :
\begin{itemize}
\item $\kappa > \frac{1}{2}$, we will say that $X$ is \textit{supercritical}
\item $\kappa = \frac{1}{2}$, we will say that $X$  is \textit{critical}
\item $\kappa < \frac{1}{2}$, we will say that $X$ is \textit{subcritical}
\end{itemize}
The aim of this paper is therefore to show the existence of quasi-limiting, $Q$-process and quasi-ergodic distribution for each regime. More precisely, it will be shown in Section 2 that, for any probability measure $\mu$ on $(-1,1)$ and $s \geq 0$, 
$$\lim_{t \to \infty} \P_{\mu,s}(X_t \in \cdot | \tau_X > t) = \delta_0$$
in the supercritical regime. This regime is of little interest and the existence of a $Q$-process and a quasi-ergodic distribution will not be shown.

The Section 3 is devoted to the critical case. This is the only regime for which there is a (unique)
quasi-stationary distribution for $X$. More precisely, it will be shown in subsection 3.1. that the conditional probability distribution $\P_{\mu,s}[X_t \in \cdot | \tau_X > t]$ converges polynomially fast in total variation to the quasi-stationary distribution, where the \textit{total variation distance} of two probability measures $\mu$ and $\nu$ is defined as 
$$\|\mu - \nu\|_{TV} := \sup_{\|f\|_\infty \leq 1} \left| \int_{(-1,1)} f(x) \mu(dx) - \int_{(-1,1)} f(x) \nu(dx) \right|.$$
Moreover, this convergence in total variation holds uniformly in the initial distribution $\mu$. This is due to the fact that, for this regime, the process $X$ is
actually obtained from an Ornstein-Uhlenbeck process by a non-linear time change. As a result,
the quasi-stationary distribution for $X$ is the one for an Ornstein-Uhlenbeck process. In a same way, the existence of a $Q$-process is shown in subsection 3.2. However, the existence and uniqueness of a quasi-ergodic
distribution, which will be dealt with in subsection 3.3, cannot be deduced from this time change and the proof requires more technical
arguments. Moreover, it is noteworthy that, contrary to what is expected, the quasi-ergodic
distribution does not coincide with the stationary measure of the Q-process.

Finally, the main part of this paper will be devoted to the subcritical regime, in section 4. In particular, it
is shown in subsection 4.3. that, for any initial law $\mu$ and any $s \geq 0$, $\P_{\mu,s}(X_t \in \cdot |\tau_X > t)$ converges weakly, when $t$ goes to infinity,
towards the quasi-stationary distribution for a Brownian motion absorbed at $\{-1, 1\}$. The key
argument is an approximation of the law of $(X_t)_{t \geq s}$ by
the one of a time changed Brownian motion, when $s$ goes to infinity. This approximation is established in the subsection 4.1, and will be also used to deduce the existence of a quasi-ergodic distribution in the subsection 4.4. The subsection 4.5 is finally concluded by showing
the existence of a $Q$-process.

Let us now introduce some notation. 
For any $E \subset \R$, one denotes by $\cM_1(E)$ the set of the probability measures supported on $E$ and, for any measurable bounded function $f$ on $(-1,1)$ and $\mu \in \cM_1((-1,1))$, denote by
$$\mu(f) := \int_{(-1,1)} f d\mu.$$ 
For a general Markov process $(A_t)_{t \geq 0}$, denote by $(\cF^A_{s,t})_{s \leq t}$ the canonical filtration of $(A_t)_{t \geq 0}$ and $(\P_{x,s}^A)_{x \in \R, s \geq 0}$ a family of probability measure such that, for any $x \in \R$ and $s \geq 0$, $\P_{x,s}^A(X_s = x) = 1$. For any probability measure $\mu$ on $\R$ and any $s \geq 0$, we define $\P^A_{\mu,s} := \int_\R \P_{x,s}\mu(dx)$. Then the family of probability measures $(\P_{\mu,s}^A)_{\mu \in \cM_1(\R), s\geq 0}$ satisfies
$$\P_{\mu,s}^A(A_s \in \cdot) = \mu.$$
If the process $A$ is time-homogeneous, we define, for any $\mu \in \cM_1(\R)$, $\P_{\mu}^A := \P_{\mu,0}^A$ and we have, for any $s \leq t$,
$$\P_{\mu,s}^A(A_{[s,t]} \in \cdot) = \P_\mu^A(A_{[0,t-s]} \in \cdot).$$
For $A = X$, we will keep the notation established at the beginning of the introduction.


\section{The supercritical regime : $\kappa > \frac{1}{2}$}
This section is devoted to the situation $\kappa > 1/2$.
The following theorem states the existence of a unique quasi-limiting distribution, which is $\delta_0$:
\begin{theorem}
For any $\mu \in \cM_1((-1,1))$ and $s \geq 0$,
\begin{equation}
\label{conv}
\lim_{t \to \infty} \P_{\mu,s}(X_t \in \cdot | \tau_X > t) = \delta_0.\end{equation}
\end{theorem}
\begin{proof}
By Markov's inequality, for any $\epsilon > 0$ and any probability measure $\mu$,
\begin{align*}\P_{\mu,s}\left(\left|X_t \right| \geq \epsilon \middle| \tau_X > t\right) &\leq \frac{\E_{\mu,s}(X^2_t|\tau_X > t)}{\epsilon^2}\\ &\leq \frac{\E_{\mu,s}(X^2_t)}{\epsilon^2 \P_{\mu,s}(\tau_X > t)} \\&=\frac{t-s + (s+1)^{2\kappa} \int_{(-1,1)} x^2 d\mu(x)}{\epsilon^2 (t+1)^{2\kappa} \P_{\mu,s}(\tau_X > t)}.\end{align*}  
Then the convergence \eqref{conv} is a consequence of the following lemma. 
\begin{lemma}
\label{lemma}
For any $s \geq 0$ and any probability measure $\mu$ on $(-1,1)$, 
$$\lim_{t \to \infty} \P_{\mu,s}(\tau_X > t) = \P_{\mu,s}(\tau_X = \infty) > 0.$$ 
\end{lemma}
\end{proof}
\begin{proof}[Proof of the lemma \ref{lemma}] In this proof, we will denote for any Markov process $A = (A_t)_{t \geq 0}$ and any positive function $f$
$$\tau_f^A := \inf\{t \geq 0 : |A_t| = f(t)\}.$$
Without loss of generality, we will show Lemma \ref{lemma} for $s = 0$. 
It is well known that, $\P_{\mu,0}$-almost surely (for any measure $\mu$),
$$\limsup_{t \to \infty} \frac{B_t}{\sqrt{2t\log \log t}} = 1,$$
and 
$$\liminf_{t \to \infty} \frac{B_t}{\sqrt{2t\log \log t}} = -1.$$
Thus, since $\kappa > \frac{1}{2}$, $\P_{\mu,0}$-almost surely,
$$X_t = \frac{B_t}{(t+1)^\kappa} \underset{t \to \infty}{\longrightarrow} 0.$$
In particular, the process $(X_t)_{t \geq 0}$ is bounded almost surely. Denote by 
$$M := \sup_{t \geq 0} |X_t|.$$
Then, for any probability measure $\mu$ on $(-1,1)$,
$$\P_{\mu,0}(\tau_X = \infty) = \P_{\mu,0}(M < 1).$$
Since $\lim_{t \to \infty} \P_{0,0}(M < t) = \P_{0,0}(M < +\infty) = 1$ and the function $t \mapsto  \P_{0,0}(M < t)$ is non-decreasing, there exists $t_0$ such that
$$c := \inf_{t \geq t_0} \P_{0,0}(M < t) > 0.$$
Let $s_0 \geq 0$ and $\mu$ a probability measure on $(-1,1)$. Then,
$$\P_{\mu,0}(\tau_X = \infty) \geq \P_{\mu,0}(\tau_X = \infty, s_0 < \tau_0 < \tau_X),$$
where $\tau_0 := \inf\{t \geq 0 : X_t = 0\}$. Then, by the strong Markov property,
$$\P_{\mu,0}(\tau_X = \infty) \geq \E_{\mu,0}(\1_{s_0 < \tau_0 < \tau_X} \P_{0,\tau_0}(\tau_X = \infty)).$$
For any $s \geq s_0$, 
$$\P_{0,s}(\tau_X = \infty) = \P_0^B(\tau^B_{t \mapsto (t + s + 1)^\kappa} = \infty),$$
However, by the scaling property of Brownian motion, the process $\fB := (\sqrt{s+1}B_{t/(s+1)})_{t \geq 0}$ is a Brownian motion and 
\begin{align*}
  \P_0^B(\tau^B_{t \mapsto (t + s + 1)^\kappa} = \infty) &=  \P_0^\fB(\tau^\fB_{t \mapsto (t + s + 1)^\kappa} = \infty) \\ 
&=  \P_0^B(\tau^B_{t \mapsto (s+1)^{\kappa-\frac{1}{2}}(t + 1)^\kappa} = \infty) \\
&= \P_{0,0}(M < (s+1)^{\kappa - \frac{1}{2}}).
\end{align*}
So, choosing $s_0$ such that $(s_0 +1)^{\kappa - \frac{1}{2}} = t_0$, one has
$$\P_{0,s}(\tau_X = \infty) \geq c,~~~~\forall s \geq s_0.$$
As a result, 
$$\P_{\mu,0}(\tau_X = \infty) \geq c \P_{\mu,0}(s_0 < \tau_0 < \tau_X) > 0.$$
\end{proof}

\section{The critical case : $\kappa = \frac{1}{2}$}

This section is devoted to the situation $\kappa = 1/2$.

\subsection{Existence and uniqueness of a quasi-stationary distribution}
\label{qsd-critical}

We state a first theorem on the existence of the quasi-limiting distribution (and quasi-stationary distribution) in the critical regime.
\begin{theorem}
\label{main-thm}
Let $\alpha_{OU}$ be the unique quasi-stationary distribution for the Ornstein-Uhlenbeck process absorbed by $(-1,1)^c$ whose generator is
\begin{equation}
\label{generator}
Lf(x) := \frac{1}{2}f''(x)  - \frac{1}{2} x f'(x),~~~~\forall x \in (-1,1), \forall f \in \cD,
\end{equation}
where $\cD := \{g \in C^2([-1,1]) : g(-1) = g(1) = 0\}$.
Then $\alpha_{OU}$ is also the unique quasi-stationary distribution for $X$ and there exist $C_{OU}, \gamma_{OU} > 0$ such that, for any probability measure $\mu$ on $(-1,1)$ and any $0 \leq s \leq t$,
\begin{equation}
\label{quasi-erg}
||\P_{\mu,s}(X_t \in \cdot | \tau_X > t) - \alpha_{OU}||_{TV} \leq C_{OU} \left(\frac{s+1}{t+1}\right)^{\gamma_{OU}}.
\end{equation}
In particular, for any $\mu \in \cM_1((-1,1))$ and $s \geq 0$, the sequence $\P_{\mu,s}(X_t \in \cdot | \tau_X > t)$ converges weakly towards $\alpha_{OU}$, when $t$ goes to infinity. 
\end{theorem}
\begin{remark}
Using the spectral theory for the Ornstein-Uhlenbeck generator, $\alpha_{OU}$ can easily be computed and one has
$$\alpha_{OU}(dx) := K \times (1-x^2)e^{-\frac{x^2}{2}}dx,$$
where $K$ is the renormalization constant. In particular, $x \mapsto (1-x^2)$ is the opposite of a Hermite polynomial which is positive on $(-1,1)$ and vanishing at $\{-1,1\}$, and $\pi(dx) := e^{-\frac{x^2}{2}}$ is a reversible measure for $L$. 
\end{remark}
\begin{remark}
 It is well known (see \cite{CMSM,MV2012}) that there exists $\lambda_{OU} > 0$ such that
\begin{equation}
\label{exit-time}
\P^Z_{\alpha_{OU}}(\tau_Z > t) = e^{-\lambda_{OU}t},~~~~\forall t \geq 0,\end{equation}
where $\tau_Z := \inf\{t \geq 0 : |Z_t| = 1\}$. Moreover, for any $f \in \cD$,
$$\alpha_{OU}(Lf) = -\lambda_{OU} \alpha_{OU}(f),$$
where $L$ is defined in \eqref{generator}. Using the explicit formula of $\alpha_{OU}$, it is easy to check that 
\begin{equation}
\label{lambda}
\lambda_{OU} = 1.
\end{equation}
\end{remark}
\begin{proof}[Proof of Theorem \ref{main-thm}]
Let $Z$ be the Ornstein-Uhlenbeck process whose infinitesimal generator is $L$. Then, for any probability measure $\mu$ on $(-1,1)$ and any $s \geq 0$,
\begin{equation}
\label{time-change}
\P_{\mu,s}((X_t)_{t \geq s} \in \cdot) = \P^Z_{\mu}\left(\left(Z_{\log\left(\frac{t+1}{s+1}\right)}\right)_{t \geq s} \in \cdot \right).\end{equation}
This can be shown using that there exists a Brownian motion $(W_t)_{t \geq 0}$ (starting from $0$) such that, for any $u \geq 0$,
$$Z_u = e^{-\frac{u}{2}} \left(Z_0 + \int_0^u e^{\frac{v}{2}} dW_v\right) = e^{-\frac{u}{2}} \left(Z_0 + \tilde{W}_{e^u-1}\right),$$
where $\tilde{W}$ is another Brownian motion starting from $0$, and setting $u = \log\left(t+1\right)$,
$$Z_{\log(1+t)} = \frac{Z_0 + \tilde{W}_t}{\sqrt{t+1}}.$$
Hence, using \eqref{time-change}, one has for any $s \leq t$,
\begin{align*}\P_{\alpha_{OU},s}(X_t \in \cdot | \tau_X > t) &= \P^Z_{\alpha_{OU}}\left(Z_{\log\left(\frac{t+1}{s+1}\right)} \in \cdot \middle| \tau_Z > \log\left(\frac{t+1}{s+1}\right)\right)\\
 &= \alpha_{OU}.\end{align*} 
In other words $\alpha_{OU}$ is also the unique quasi-stationary distribution for the time-inhomogeneous Markov process $X$. Moreover, since $Z$ satisfies the assumptions $(A1)$ and $(A2)$ of \cite{CV2014} (this is actually shown in \cite{CV2017b}), then, by Theorem 2.1. in \cite{CV2014}, there exist $C_{OU} > 0$ and $\gamma_{OU} > 0$ such that, for any $t \geq 0$ and for any probability measure $\mu$,
$$||\P^Z_{\mu}(Z_t \in \cdot | \tau_Z > t) - \alpha_{OU}||_{TV} \leq C_{OU} e^{-\gamma_{OU} t}.$$
Using \eqref{time-change}, one deduces that, for any $s \leq t$ and for any probability measure $\mu$ on $(-1,1)$,
\begin{equation*}
||\P_{\mu,s}(X_t \in \cdot | \tau_X > t) - \alpha_{OU}||_{TV} \leq C_{OU} \left(\frac{s+1}{t+1}\right)^{\gamma_{OU}}.
\end{equation*}
This concludes the proof.
\end{proof}

\subsection{Existence of the $Q$-process}
Before tackling the existence of the $Q$-process, we need the following proposition:

\begin{proposition}
\label{1}
There exists a non-negative function $\eta_{OU} : [-1,1] \to \R_+$, positive on $(-1,1)$ and vanishing on $\{-1,1\}$,  such that, for any $x \in (-1,1)$ and any $s \geq 0$,
$$\eta_{OU}(x) = \lim_{t \to \infty} \frac{t+1}{s+1} \P_{x,s}(\tau_X > t),$$
where the convergence holds uniformly on $[-1,1]$ and $\alpha_{OU}(\eta_{OU})=1$. Moreover the function $\eta_{OU}$ is bounded, belongs to the domain of $L$ defined in \eqref{generator}, and 
$$L \eta_{OU} = - \lambda_{OU} \eta_{OU} = - \eta_{OU},$$   
where $\lambda_{OU}$ is defined in Remark 2.
\end{proposition}

\begin{remark} More precisely,
\begin{equation}
\label{eigenfunction1}
\eta_{OU}(x) = K'\times (1-x^2),
\end{equation}  
where $K'$ is the positive constant such that $\alpha_{OU}(\eta_{OU})=1$.
\end{remark}
 An interesting consequence of Proposition
1 and (6) is stated as the following corollary:
\begin{corollary}
\label{cor}
Let $B$ a Brownian motion on $\R$, and denote by 
$$\tau_B^{\sqrt{\cdot}} := \inf\{t \geq 0 : |B_t| \geq \sqrt{t+1}\}.$$
Then, for any $x \in (-1,1)$, 
$$\P_x^B(\tau_B^{\sqrt{\cdot}} > t) \sim_{t \to \infty} K' \frac{1-x^2}{t+1}.$$
\end{corollary}

\begin{proof}[Proof of Proposition \ref{1} and Corollary \ref{cor}] 
Using Proposition 2.3 in \cite{CV2014} applied to the process $Z$ and \eqref{time-change}, one has, for any $x \in (-1,1)$ and $s \geq 0$,
\begin{align*}
\eta_{OU}(x) &= \lim_{t \to \infty} e^{\lambda_{OU} \log\left(\frac{t+1}{s+1}\right)} \P_x^Z\left(\tau_Z > \log\left(\frac{t+1}{s+1}\right)\right) \\
&=  \lim_{t \to \infty} \left(\frac{t+1}{s+1}\right)^{\lambda_{OU}} \P_{x,s}(\tau_X > t) \\
&=  \lim_{t \to \infty} \frac{t+1}{s+1} \P_{x,s}(\tau_X > t).
\end{align*}
where we finally used \eqref{lambda}. This ends the proof of Proposition \ref{1}.  Now it is easy to see that, for any $x \in (-1,1)$ and $t \geq 0$, $\P_x^B(\tau_B^{\sqrt{\cdot}} > t) = \P_{x,0}(\tau_X > t)$. Thus, using Proposition \ref{1} and \eqref{eigenfunction1}, we conclude the corollary. 
\end{proof} 
\begin{remark}
In \cite{breiman1967}, Breiman shows a similar result for one-dimensional Brownian motion absorbed by a one-sided square boundary. More precisely, denoting $T^*_c := \inf\{t \geq 0 : B_t \geq c \sqrt{t+1}\}$ for any $c > 0$, he shows that $\P^B_{0}(T^*_c > t) \sim_{t \to \infty} a t^{-b(c)}$ for $a > 0$ and $b$ such that $b(1) = 1$. In particular, for $c=1$, $\P_0^B(T^*_1 > t)$ and $\P_0^B(\tau_B^{\sqrt{\cdot}} > t)$ decay as $1/t$. The reader can also see \cite{salminen1988} for more general boundaries.
\end{remark} 
We turn to the existence of the $Q$-process and its ergodicity. 
\begin{proposition}
\label{q-proc-prop}
\begin{itemize}
\item There exists a $Q$-process and the family of probability measures $(\Q_{x,s})_{x \in (-1,1), s \geq 0}$ defined in Definition \ref{qproc-def} is given by : for any $x \in (-1,1)$ and $s \leq t$, 
\begin{align*}\Q_{x,s}(X_{[s,t]} \in \cdot) &= \E_{x,s}\left(\1_{X_{[s,t]} \in \cdot, \tau_X > t} \left(\frac{t+1}{s+1}\right)^{\lambda_{OU}} \frac{\eta_{OU}(X_t)}{\eta_{OU}(x)}\right) \\
&= \frac{t+1}{s+1} \times \E_{x,s}\left(\1_{X_{[s,t]} \in \cdot, \tau_X > t} \frac{1-X_t^2}{1-x^2}\right).\end{align*}
\item The probability measure $\beta_{OU}$ defined by 
$$\beta_{OU}(dx) := \eta_{OU}(x) \alpha_{OU}(dx) = KK' (1-x^2)^2 e^{-\frac{x^2}{2}}dx$$
is the unique stationary distribution of $X$ under $(\Q_{x,s})_{s \geq 0, x\in (-1,1)}$. Moreover, for any $0 \leq s \leq t$ and any $x \in (-1,1)$, 
$$||\Q_{x,s}(X_t \in \cdot) - \beta_{OU}||_{TV} \leq C_{OU} \left(\frac{s+1}{t+1} \right)^{\gamma_{OU}}, $$
where $C_{OU}$ and $\gamma_{OU}$ are the same constant as used in \eqref{quasi-erg}.
\end{itemize}
\end{proposition}
\begin{proof}
Straightforward using \eqref{time-change} and Proposition 3.1 in \cite{CV2014} applied to the Ornstein-Uhlenbeck process $Z$. 
\end{proof}

\subsection{Quasi-ergodic distribution}
Now let us provide and show the existence and the uniqueness of the quasi-ergodic distribution:
\begin{theorem}
\label{quasi-ergodic-critical}
For any probability measure $\mu$ on $(-1,1)$ and any $s \geq 0$, for any measurable set $S$, 
$$\lim_{t \to \infty} \frac{1}{t} \int_s^t \P_{\mu,s}(X_u \in S | \tau_X > t)du = \int_S \E^Z_{x}\left(\tau_Z \right) \alpha_{OU}(dx),$$
where we recall that $Z$ is the Ornstein-Uhlenbeck process whose the generator is \eqref{generator}.
\end{theorem}
\begin{remark}
As mentioned in the introduction, the quasi-ergodic distribution $\E^Z_x(\tau_Z) \alpha_{OU}(dx)$ is different from the invariant measure of the $Q$-process $\beta_{OU}$. According to our knowledge, there does not exist any explicit formula for the density $x \mapsto \E_x^Z(\tau_Z)$. However, noting that $\beta_{OU}(dx) = \eta_{OU}(x) \alpha_{OU}(dx)$ where $\eta_{OU}$ is defined as 
$$L \eta_{OU} = - \eta_{OU},$$
$\beta_{OU}$ is surely different from $\E^Z_x(\tau_Z) \alpha_{OU}(dx)$ since
$$L \E_\cdot^Z(\tau_Z) = -1.$$
\end{remark}
\begin{proof}
In order to make the proof easier to read, Theorem \ref{quasi-ergodic-critical} will be proved for $s=0$ in what follows. The proof for a general $s$ will be very similar to the following one.
  
First, using the variable change $u = qt$, one has, for any $\mu \in \cM_1((-1,1))$, $t > 0$ and $f$ continuous and bounded measurable,
$$\frac{1}{t} \int_0^t \E_{\mu,0}(f(X_u) | \tau_X > t)du = \int_0^1 \E_{\mu,0}(f(X_{qt}) | \tau_X > t)dq.$$
As a result, it is enough to show the weak convergence of $\left(\P_{\mu,0}(X_{qt} \in \cdot | \tau_X > t)\right)_{t \geq 0}$ for any $q \in (0,1)$, then to conclude with the Lebesgue's dominated convergence theorem.\\
Let  $\mu \in \cM_1((-1,1))$, $q \in (0,1)$ and $f$ continuous and bounded measurable. By the Markov property and \eqref{time-change}, for any $t \geq 0$,
\begin{align}
 \E_{\mu,0}(f(X_{qt})\1_{\tau_X > t}) &= \E_{\mu,0}\left(f(X_{qt}) \1_{\tau_X > qt} \P_{X_{qt},qt}(\tau_X > t)\right) \notag \\
&= \E_{\mu,0}\left(f_t\left(X_{qt}\right) \1_{\tau_X > qt}\right),
\label{ft2}
\end{align}
where we set for any $y \in (-1,1)$,
$$f_t(y) := f(y)  \P_{y,qt}\left[\tau_X > t\right].$$
By \eqref{time-change}, for any $y \in (-1,1)$ and $t \geq 0$,
\begin{align*}
f_t(y) &= f(y)\P^Z_{y}\left[\tau_Z > \log\left(\frac{t+1}{qt+1}\right)\right].
\end{align*} 
Now define for any $y \in (-1,1)$,
$$f_\infty(y) := f(y)  \P^Z_{y}\left[\tau_Z > -\log\left(q\right)\right].$$
It is easy to see that $(f_t)_{t \geq 0}$ converges pointwise towards $f_\infty$. Moreover, a simple calculus computation entails that the function $t \to \frac{t+1}{qt+1}$ is increasing, which implies that the sequence $(f_t)_{t \geq 0}$ is a decreasing sequence of continuous functions defined on $[-1,1]$. Likewise, $f_\infty$ is continuous on $[-1,1]$. As a result, by Dini's theorem for the decreasing sequences of continuous functions, the pointwise convergence is equivalent to the uniform convergence on $[-1,1]$. Thus, 
\begin{equation}
\label{unif}\lim_{t \to \infty} \sup_{y \in (-1,1)} |f_t(y) - f_\infty(y)| = 0.\end{equation}
Now, let us show that
\begin{equation}
\label{cequelonve}
\lim_{t \to \infty}(qt+1) \E_{\mu,0}(f_\infty(X_{qt})\1_{\tau_X >qt}) = \mu(\eta_{OU}) \alpha_{OU}(f_\infty).
\end{equation}
To show this, let us begin with
\begin{align*}
&(qt+1) \E_{\mu,0}\left(f_\infty(X_{qt})\1_{\tau_X > qt}\right) =(qt+1)  \P_{\mu,0}(\tau_X > qt) \times \E_{\mu,0}(f_\infty(X_{qt})|{\tau_X > qt}).
\end{align*}
 On the one hand, by Proposition \ref{1}, 
 $$\lim_{t \to \infty} (qt+1)  \P_{\mu,0}(\tau_X > qt) = \mu(\eta_{OU}).$$
 On the other hand, by \eqref{quasi-erg}, 
 $$\lim_{t \to \infty} \E_{\mu,0}(f_\infty(X_{ qt})|{\tau_X > qt}) = \alpha_{OU}(f_\infty).$$
\eqref{cequelonve} follows from these two convergences. Now, by \eqref{cequelonve} and \eqref{unif},    
\begin{align*}
   (qt+1)  \E_{\mu,0}\left(f_t\left(X_{qt}\right) \1_{\tau_X > qt}\right) & = (qt+1)  \E_{\mu,0}\left(f_\infty\left(X_{qt}\right) \1_{\tau_X > qt}\right) \medskip \\&~~~~~~~~~~~~~~~~~~~~ + (qt+1)  \E_{\mu,0}\left(\left[f_\infty\left(X_{qt}\right) - f_t\left(X_{qt}\right)\right] \1_{\tau_X > qt}\right) \medskip \\ & \underset{t \to \infty}{\longrightarrow} \mu(\eta_{OU}) \alpha_{OU}(f_\infty),
\end{align*}
because
\begin{multline*}\left|(qt+1)  \E_{\mu,0}\left[\left(f_\infty\left(X_{qt}\right) - f_t\left(X_{qt}\right)\right) \1_{\tau_X > qt}\right]\right|  \leq  (qt+1)  \P_{\mu,0}(\tau_X > qt) \times  \sup_{y \in (-1,1)} |f_t(y) - f_\infty(y)| \underset{t \to \infty}{\longrightarrow} 0.\end{multline*}
Hence, using \eqref{ft2},
\begin{align*}\lim_{t \to \infty} (qt+1) \E_{\mu,0}(f(X_{qt})\1_{\tau_X > t})&= \mu(\eta_{OU}) \alpha_{OU}(f_\infty)\\
&= \mu(\eta_{OU}) \int_{(-1,1)} f(x) \P^Z_{x}(\tau_Z > -\log(q)) \alpha_{OU}(dx).\end{align*}
Moreover, taking $f = \1$, using \eqref{exit-time} and \eqref{lambda},
\begin{align*}\lim_{t \to \infty} (qt+1) \P_{\mu,0}({\tau_X > t})
&= \mu(\eta_{OU}) \P^Z_{\alpha_{OU}}(\tau_Z > -\log(q)) \\
&= \mu(\eta_{OU}) q.\end{align*}
Thus, we deduce that 
$$\lim_{t \to \infty}   \E_{\mu,0}(f(X_{qt}) | {\tau_X > t}) =  q^{-1} \int_{(-1,1)} \alpha_{OU}(dx) f(x)\P_x(\tau_Z > -\log(q)).$$ 
Then, by Lebesgue's theorem, for any probability measure $\mu$ on $(-1,1)$ and any bounded measurable function $f$,
\begin{align*}
\lim_{t \to \infty} \frac{1}{t} \int_0^t \E_{\mu,0}(f(X_u)|{\tau_X>t})du &= \lim_{t \to \infty} \int_0^1 \E_{\mu,0}(f(X_{qt})|{\tau_X>t})dq \\
&= \int_0^1 q^{-1} \int_{(-1,1)} f(x) \P^Z_x(\tau_Z> -\log(q)) \alpha_{OU}(dx) dq \\
&= \int_{(-1,1)} \alpha_{OU}(dx) f(x) \int_0^1 q^{-1} \P^Z_x(\tau_Z > -\log(q))dq \\
&=  \int_{(-1,1)} \alpha_{OU}(dx) f(x) \int_0^\infty  \P^Z_x(\tau_Z > s)ds \\
&= \int_{(-1,1)} \alpha_{OU}(dx) f(x) \E^Z_x\left(\tau_Z\right). 
\end{align*}
This concludes the proof.
\end{proof}
\begin{remark}
As it is seen in the previous proof, the quasi-ergodic distribution for $X$ is obtained computing the limit of $\P_{\mu,s}(X_{qt} \in \cdot | \tau_X > t)$, when $t$ goes to infinity and for $q \in (0,1)$ fixed. By the time change $t \mapsto \log(1+t)$, this limit is the same as the one of 
$$\P^Z_\mu(Z_{\log(qt)} \in \cdot | \tau_Z > \log(t)) = \P^Z_\mu(Z_{\log(q) + \log(t)} \in \cdot | \tau_Z > \log(t)),$$
with $\log(q) < 0$. Such a limiting probability measure is called a $-\log(q)$-Yaglom limit and is different from the invariant measure of the $Q$-process of $Z$ (obtained taking $q = +\infty$). This provides a heuristic reason explaining why the quasi-ergodic distribution for $X$ is different from the one for the Ornstein-Uhlenbeck process $Z$.  
\end{remark} 
\section{The subcritical case : $\kappa < \frac{1}{2}$.} 
In this section, we will show that a quasi-limiting distribution, quasi-ergodic distribution and a $Q$-process exist when $\kappa < \frac{1}{2}$.   
To do this, the strategy will be to compare (in a sense described later) the process $X$ to the process $Y$ defined by  
$$Y_t := \int_0^t \frac{1}{(u+1)^\kappa} dB_u,~~~~\forall t \geq 0.$$
Then the quasi-stationarity for $X$ will be deduced from the one for $Y$. 
\subsection{Approximation by $Y$ through asymptotic pseudotrajectories}
Denote by $\tau_Y := \inf\{t \geq 0 : |Y_t| = 1\}$. The aim of this subsection is to show the following proposition: 
\begin{proposition} 
\label{asymptotic-pseudotrajectory}
There exists a function $F : \R_+ \to \R_+$ such that
$$\lim_{s \to \infty} F(s) = 0,$$
and such that, for any $0 \leq s \leq t \leq T$, for any probability measure $\mu$ on $(-1,1)$,

\begin{equation}
\label{uniform-convergence}
 ||\P_{\mu,s}(X_t \in \cdot | \tau_{X} > T) - \P^Y_{\mu,s}(Y_t \in \cdot | \tau_{Y} > T)||_{TV} \leq F(s).
\end{equation}
\end{proposition}
\begin{remark}   
\eqref{uniform-convergence} provides us with a decay towards $0$ uniformly in the initial law, $t$ and $T$. It can be seen as an analogue of the asymptotic pseudotrajectories introduced by Benaïm and Hirsch in \cite{BH1996}. See also \cite{benaim1999} for more details about asymptotic pseudotrajectories in the general case.
\end{remark}
\begin{proof}[Proof of Proposition \ref{asymptotic-pseudotrajectory}]
By Itô's formula, one has, for any $t \geq 0$,
$$X_t = X_0 + Y_t - \int_0^t \kappa (s+1)^{-\kappa - 1}B_s ds =  X_0 + Y_t - <Y,M>_t,$$
where
$$M_{t} := \int_0^t \kappa(u+1)^{\kappa-1}X_udB_u =  \int_0^t \kappa(u+1)^{2\kappa-1}X_udX_u +  \int_0^t (\kappa(u+1)^{\kappa-1}X_u)^2du.$$
For any $s \leq t$, denote by $$M_{s,t} := M_t - M_s = \int_s^t \kappa(u+1)^{\kappa-1}X_udB_u,$$ and $$<M,M>_{s,t} := <M,M>_t - <M,M>_s = \int_s^t (\kappa(u+1)^{\kappa-1}X_u)^2du.$$
Moreover, denote by $\mathcal{E}(M)_{s,t}$ the exponential local martingale defined by
\begin{align*}
\mathcal{E}(M)_{s,t}& := \exp\left(M_{s,t} - \frac{1}{2}<M,M>_{s,t}\right)\\ &= \exp\left(\int_s^t \kappa(u+1)^{2\kappa-1}X_udX_u + \frac{1}{2}\int_s^t (\kappa(u+1)^{\kappa-1}X_u)^2du\right) \\
&= \exp\left(\frac{1}{2}N_{s,t}\right),
\end{align*}
where $$N_{s,t} := 2 \int_s^t \kappa(u+1)^{2\kappa-1}X_udX_u + \int_s^t (\kappa(u+1)^{\kappa-1}X_u)^2du.$$ By It\^{o}'s formula applied to $t \to \kappa(t+1)^{2\kappa-1}X_t^2$, for any $s \leq t$, 
\begin{align*}
N_{s,t} &= \kappa(t+1)^{2\kappa-1}X_t^2 - \kappa(s+1)^{2\kappa-1}X_s^2 \\&~~~~~~~~~~~~- \int_s^t [\kappa(u+1)^{2\kappa-1}]'X^2_udu - \int_s^t \frac{\kappa}{u+1}du + \int_s^t (\kappa(u+1)^{\kappa-1}X_u)^2du \\
&= \kappa(t+1)^{2\kappa-1}X_t^2 - \kappa(s+1)^{2\kappa-1}X_s^2 \\&~~~~~~~~~~~~- \int_s^t [\kappa(u+1)^{2\kappa-1}]'X^2_udu - \kappa \log \left(\frac{t+1}{s+1}\right) + \int_s^t (\kappa(u+1)^{\kappa-1}X_u)^2du.
\end{align*}
Note that the process $(N_{s,t \land \tau_X})_{s \leq t}$ is almost surely uniformly bounded, thus $\mathcal{E}(M)_{s,t \land \tau_X}$ is a martingale.  
For any $t \geq s \geq 0$ and $\mu \in \cM_1((-1,1))$, define $\G_{\mu,s}$ the probability measure satisfying
$$\G_{\mu,s}(A) = \E_{\mu,s}(\mathcal{E}(M)_{s,t \land \tau_X} \1_{A}), ~~~~\forall A \in \sigma(X_u, s \leq u \leq t).$$
Then, by the Girsanov theorem, the law of $(X_{t \land \tau_X})_{t \geq s}$ under $\G_{\mu,s}$ is the law of $(Y_{t \land \tau_Y})_{t \geq s}$ under $\P_{\mu,s}$. In particular, for any $S$ measurable set, probability measure $\mu$ on $(-1,1)$ and $0 \leq s \leq t \leq T$,
\begin{align}
    \P^Y_{\mu,s}(Y_t \in S, \tau_Y > T) &= \G_{\mu,s}(X_t \in S, \tau_X > T) \notag \\
    &= \E_{\mu,s}\left(\mathcal{E}(M)_{s,T \land \tau_X} \1_{X_t \in S, \tau_X > T}\right) \notag \\
    &= \E_{\mu,s}\left(\mathcal{E}(M)_{s,T} \1_{X_t \in S, \tau_X > T}\right) \notag \\
    &= \E_{\mu,s}\left(\exp\left(\frac{1}{2}N_{s,T}\right) \1_{X_t \in S, \tau_X > T}\right) \notag \\
&= \left( \frac{s+1}{T+1}\right)^{\frac{\kappa}{2}} \E_{\mu,s}\left(\exp\left(\frac{1}{2}N'_{s,T}\right) \1_{X_t \in S, \tau_X > T}\right),
\label{rev}
\end{align}
with $N'_{s,T} = N_{s,T} + \kappa \log \left(\frac{T+1}{s+1}\right)$. Thus,
$$\P^Y_{\mu,s}(Y_t \in S | \tau_Y >T) = \frac{\E_{\mu,s}\left(\exp\left(\frac{1}{2}N'_{s,T}\right) \1_{X_t \in S,\tau_X>T}\right)}{\E_{\mu,s}\left(\exp\left(\frac{1}{2}N'_{s,T}\right)\1_{\tau_X>T}\right)}.$$
By this last inequality and by a tringular inequality, one has, for any $0 \leq s \leq t \leq T$ and for any measurable set $S$,
\begin{align*}
&|\P_{\mu,s}(X_t \in S | \tau_{X} > T) - \P^Y_{\mu,s}(Y_t \in S | \tau_{Y} > T)|\\
&\leq  \left| \frac{\P_{\mu,s}(X_t \in S, \tau_{X} > T)}{\P_{\mu,s}(\tau_{X} > T)} - \frac{\E_{\mu,s}\left(\exp\left(\frac{1}{2}N'_{s,T}\right) \1_{X_t \in S,\tau_X>T}\right)}{\P_{\mu,s}(\tau_{X} > T)} \right| \\
&~~~~~~~~~~~~~~~~~~~~~~~~+  \left|  \frac{\E_{\mu,s}\left(\exp\left(\frac{1}{2}N'_{s,T}\right) \1_{X_t \in S,\tau_X>T}\right)}{\P_{\mu,s}(\tau_{X} > T)} - \frac{\E_{\mu,s}\left(\exp\left(\frac{1}{2}N'_{s,T}\right) \1_{X_t \in S,\tau_X>T}\right)}{\E_{\mu,s}\left(\exp\left(\frac{1}{2}N'_{s,T}\right)\1_{\tau_X>T}\right)}  \right| \\
&\leq  \underbrace{\frac{| \P_{\mu,s}(X_t \in S, \tau_X > T) - \E_{\mu,s}\left(\exp\left(\frac{1}{2}N'_{s,T}\right) \1_{X_t \in S,\tau_X>T}\right)|}{\P_{\mu,s}(\tau_{X} > T)}}_{=:  ~C_{s}(\mu,t,T,S)} \\ &~~~~~~~~~~~~~~~~~~~~~~~~+  \frac{\E_{\mu,s}\left(\exp\left(\frac{1}{2}N'_{s,T}\right)\1_{\tau_X>T}\right)}{\P_{\mu,s}(\tau_X>T)}  \times \underbrace{\left|1- \frac{\P_{\mu,s}(\tau_{X} > T)}{\E_{\mu,s}\left(\exp\left(\frac{1}{2}N'_{s,T}\right)\1_{\tau_X>T}\right)} \right|}_{=: A_{s}(\mu,T)}.
\end{align*}
In order to show \eqref{uniform-convergence}, we will bound the functions $A_s$ and $C_s$.
\begin{description}
\item[Step $1$ : Upper bound for $C_s$]. \newline

For any $0 \leq s \leq t \leq T$, probability measure $\mu$ and $S$ measurable set, 
\begin{align*}C_s(\mu,t,T,S) &= \left| \underbrace{\E_{\mu,s}\left[\left(\exp\left(\frac{1}{2}N'_{s,T}\right)-1\right) \1_{X_t \in S} \middle| \tau_X > T \right]}_{=: f(s,t,T,\mu,S)}\right|. \end{align*}
On the event $\{\tau_X > T\}$, $X_u^2 < 1$ for any $0 \leq u \leq T$. Hence, the function $f$ defined as above is bounded as follows:
\begin{multline*}\exp\left[- \frac{\kappa}{2}(s+1)^{2\kappa-1}-\frac{1}{2}\int_s^T [\kappa(u+1)^{2\kappa-1}]'du \right]-1 \\ \leq f(s,t,T,\mu,S) \leq \exp\left[\frac{1}{2} \kappa(T+1)^{2\kappa-1} + \frac{1}{2}\int_s^T (\kappa(u+1)^{\kappa-1})^2du\right]-1.\end{multline*}
In particular, for any $0 \leq s \leq t \leq T$, for any probability measure $\mu$ and $S$ measurable set,  
\begin{align*}|f(s,t,T,\mu,S)| \leq &\left[1 - \exp\left(-\frac{1}{2}\kappa(s+1)^{2\kappa-1} \right)\right] \lor \left[\exp\left(\frac{1}{2} \left(\kappa + \frac{\kappa^2}{1-2\kappa}\right)(s+1)^{2\kappa-1}\right)-1\right] \\
& =: \phi(s).\end{align*}
Hence,
$$C_s(\mu,t,T,S) \leq \phi(s).$$
\item[Step $2$ : Upper bound for $A_s$]. \newline
Taking $S = (-1,1)$, 
\begin{align*}
C_s(\mu,t,T,(-1,1)) &=   \frac{|\E_{\mu,s}\left(\exp\left(\frac{1}{2}N'_{s,T}\right) \1_{X_t \in (-1,1),\tau_X>T}\right) - \P_{\mu,s}(X_t \in (-1,1), \tau_X > T)|}{\P_{\mu,s}(\tau_{X} > T)} \\
 &= \left|\frac{\E_{\mu,s}\left(\exp\left(\frac{1}{2}N'_{s,T}\right) \1_{\tau_X>T}\right)}{\P_{\mu,s}(\tau_X > T)} - 1 \right|. \end{align*}
According to the previous bound we have shown, for any 
 for any $s \leq T$, for any probability measure $\mu$ on $(-1,1)$, 
\begin{equation}
    \label{cpasmal}
    1-\phi(s) \leq \frac{\E_{\mu,s}\left(\exp\left(\frac{1}{2}N'_{s,t}\right) \1_{\tau_X>t}\right)}{\P_{\mu,s}(\tau_X > t)} \leq 1 + \phi(s).
\end{equation}
We deduce from this last inequality that
$$A_s(\mu,T) \leq \left(1-\frac{1}{1+\phi(s)}\right) \lor \left(\frac{1}{1-\phi(s)}-1\right) =: \psi(s).$$
\end{description}
We set then, for any $s \geq 0$,
$$F(s) = \phi(s) + \psi(s)(1+\phi(s)),$$ which concludes the proof.
\end{proof}

\subsection{Quasi-stationarity for $Y$}
Now, we are interested in the quasi-stationarity for the process $Y$. Note that,
by the Dubins-Schwartz theorem, there exists a Brownian motion $\tilde{B}$ such that, for any $t \geq 0$,
\begin{equation}
\label{ds}
Y_t  = \tilde{B}_{\frac{(t+1)^{1-2\kappa}-1}{1-2\kappa}}.\end{equation}
Denote $\tau_{\tilde{B}} := \inf\{t \geq 0 : |\tilde{B}_t| = 1\}$.
Then, by \eqref{ds}, for any initial law $\mu$ and $s \geq 0$,
\begin{align*}
    \P^Y_{\mu,s}(Y_t \in \cdot | \tau_Y > t) &= \P^{\tilde{B}}_{\mu}\left(\tilde{B}_{\frac{(t+1)^{1-2\kappa}-(s+1)^{1-2\kappa}}{1-2\kappa}} \in \cdot \middle| \tau_{\tilde{B}} > \frac{(t+1)^{1-2\kappa}-(s+1)^{1-2\kappa}}{1-2\kappa}\right).
\end{align*}
It is well known that a Brownian motion absorbed at $(-1,1)^c$ admits a unique quasi-stationary distribution $\alpha_{Bm}$, whose the explicit formula is
$$\alpha_{Bm}(dx) := \frac{1}{2} \cos\left(\frac{\pi}{2} x\right)dx,$$
and that there exists $\lambda_{Bm} > 0$ (see \cite{MV2012}) such that
$$\P_{\alpha_{Bm}}^{\tilde{B}}(\tau_Y > t) = e^{-\lambda_{Bm}t},~~~~\forall t \geq 0.$$
Remark that $\lambda_{Bm}$ satisfies also   
$$\alpha_{Bm}\left( \frac{1}{2} \Delta f\right) = - \lambda_{Bm} \alpha_{Bm}(f),~~~~\forall f \in \{g \in \cC^2([-1,1]) : g(1) = g(-1) = 0\},$$
and $\lambda_{Bm} = \frac{\pi^2}{8}$.  
The Brownian motion absorbed at $(-1,1)^c$ satisfies the Champagnat-Villemonais condition $(A1)-(A2)$ in \cite{CV2014}, which implies the existence of $C_{Bm},\gamma_{Bm} > 0$ such that, for any probability measure $\mu$ and any $t \geq 0$,
$$||\P^{\tilde{B}}_\mu(\tilde{B}_t \in \cdot | \tau_{\tilde{B}} > t) - \alpha_{Bm}||_{TV} \leq C_{Bm} e^{-\gamma_{Bm} t}.$$ 
Thus, using the Dubins-Schwartz transformation, for any $s \leq t$ and any probability measure $\mu$
\begin{equation}
\label{exponential-decay}
||\P^Y_{\mu,s}(Y_t \in \cdot | \tau_{Y} > t) - \alpha_{Bm}||_{TV} \leq C_{Bm} \exp\left(-\gamma_{Bm} \times \frac{(t+1)^{1-2\kappa}-(s+1)^{1-2\kappa}}{1-2\kappa}\right).
\end{equation}
Moreover, let $\eta_{Bm}$ be the function defined by 
$$\eta_{Bm}(x) := \lim_{t \to \infty} e^{\lambda_{Bm}t} \P_x^{\tilde{B}}(\tau_{\tilde{B}} > t) = \frac{1}{8}\cos\left(\frac{\pi x}{2}\right).$$ 
This definition makes sense by Proposition 2.3. in \cite{CV2014}. We recall moreover that $\eta_{Bm}$ is positive on $(-1,1)$, vanishing on $\{-1,1\}$, $\alpha_{Bm}(\eta_{Bm}) = 1$ and the convergence holds uniformly on $[-1,1]$. Then, 
in the same way as in the critical case, an analogous version of Propositions \ref{1} and \ref{q-proc-prop} can be stated as follows: 
\begin{proposition}
\begin{enumerate}[(i)]
\label{prop}
\item
\label{1.prop}  For any $x \in (-1,1)$ and any $s \geq 0$,
$$\eta_{Bm}(x) = \lim_{t \to \infty} e^{\lambda_{Bm} \frac{(t+1)^{1-2\kappa} - (s+1)^{1-2\kappa}}{1-2\kappa}} \P^Y_{x,s}(\tau_Y > t),$$
where the convergence holds uniformly on $[-1,1]$. 
\item There exists a $Q$-process for $Y$ in the sense of Definition \ref{qproc-def} and the family of probability measure $(\Q_{x,s}^Y)_{x \in (-1,1), s \geq 0}$ defined by $\Q_{s,x}^Y(Y_{[s,t]} \in \cdot) := \lim_{T \to \infty} \P^Y_{x,s}(Y_{[s,t]} \in \cdot | T < \tau_Y)$ satisfies also 
$$ \Q_{x,s}^Y(Y_{[s,t]} \in \cdot) = \E_{x,s}\left(\1_{Y_{[s,t]} \in \cdot, t < \tau_Y} e^{\lambda_{Bm} \frac{(t+1)^{1-2\kappa} - (s+1)^{1-2\kappa}}{1-2\kappa}} \frac{\eta_{Bm}(Y_t)}{\eta_{Bm}(x)}\right),$$
for any $x \in (-1,1)$ and $s \leq t$.
\item The probability measure $\beta_{Bm}$ defined by
\begin{equation}
    \label{quasi-ergodic}
    \beta_{Bm}(dx) = \eta_{Bm}(x) \alpha_{Bm}(dx)
\end{equation}
is the unique stationary distribution of $Y$ under $(\Q_{x,s}^Y)_{x \in (-1,1),s \geq 0}$ and, for any $x \in (-1,1)$ and $s \geq 0$,
$$||\Q_{x,s}^Y(Y_t \in \cdot) - \beta_{Bm}||_{TV} \leq  C_{Bm} \exp\left(-\gamma_{Bm} \times \frac{(t+1)^{1-2\kappa}-(s+1)^{1-2\kappa}}{1-2\kappa}\right),$$
where $C_{Bm}$ and $\gamma_{Bm}$ are the same as in \eqref{exponential-decay}.
\end{enumerate}
\end{proposition}
\begin{proof}
The proof is essentially the same as for the proof of Proposition \ref{q-proc-prop}.
\end{proof}
\subsection{Quasi-limiting distribution of $X$}
Now we will use Proposition \ref{asymptotic-pseudotrajectory} in order to show the existence of a quasi-limiting distribution for the process $X$. 
\begin{theorem}
\label{thm}
 For any probability measure $\mu$ on $(-1,1)$ and any $0 \leq s \leq t$,
\begin{equation}
\label{inequality}
||\P_{\mu,s}(X_t \in \cdot | \tau_X > t) - \alpha_{Bm}||_{TV} \leq  F\left(\frac{t}{2}\right) + C_{Bm} \exp\left(-\gamma_{Bm} \times \frac{(t+1)^{1-2\kappa}-(\frac{t}{2}+1)^{1-2\kappa}}{1-2\kappa}\right), \end{equation}
where the function $F$ is defined in Proposition \ref{asymptotic-pseudotrajectory}.
In particular, for any $\mu \in \cM_1((-1,1))$ and any $s \geq 0$,
$$\lim_{t \to \infty} \P_{\mu,s}(X_t \in \cdot | \tau_X > t) = \alpha_{Bm}.$$
\end{theorem}
\begin{proof}
Let $\mu \in \cM_1((-1,1))$. For any $s \leq t$ define 
\begin{equation*}
\mu_{(s,t)} := \P_{\mu,s}(X_t \in \cdot | \tau_X > t).\end{equation*}
Then, by the Markov property, for any $s \leq t \leq u$,
$$\mu_{(s,u)} = \P_{\mu_{(s,t)},t}(X_u \in \cdot | \tau_X > u).$$
Thus, for any $s \leq t$,
\begin{align*}
||\mu_{(s,2t)} - \alpha_{Bm}||_{TV} &\leq  ||\mu_{(s,2t)} -  \P^Y_{\mu_{(s,t)},t}(Y_{2t} \in \cdot | \tau_Y > 2t)||_{TV} \\&~~~~~~~~~~~~~~ + ||\P^Y_{\mu_{(s,t)},t}(Y_{2t} \in \cdot | \tau_Y > 2t) -  \alpha_{Bm}||_{TV} \\
&= ||\P_{\mu_{(s,t)},t}(X_{2t} \in \cdot | \tau_X > 2t) -  \P^Y_{\mu_{(s,t)},t}(Y_{2t} \in \cdot | \tau_Y > 2t)||_{TV}\\&~~~~~~~~~~~~~~ + ||\P^Y_{\mu_{(s,t)},t}(Y_{2t} \in \cdot | \tau_Y > 2t) -  \alpha_{Bm}||_{TV}\\
&\leq F(t) + C_{Bm} \exp\left(-\gamma_{Bm} \times \frac{(2t+1)^{1-2\kappa}-(t+1)^{1-2\kappa}}{1-2\kappa}\right),
\end{align*}
 where we used the inequalities \eqref{uniform-convergence} and \eqref{exponential-decay}. This shows the inequality \eqref{inequality}. 

Now, since $\lim_{t \to \infty} F(t) = 0$  by Proposition \ref{asymptotic-pseudotrajectory}, and noting that
$$\lim_{t \to \infty} \exp\left(-\gamma_{Bm} \times \frac{(2t+1)^{1-2\kappa}-(t+1)^{1-2\kappa}}{1-2\kappa}\right) = 0,$$
because $\kappa < \frac{1}{2}$,  this shows that, for any $\mu \in \cM_1((-1,1))$ and $s\geq 0$, 
$$\lim_{t \to \infty} \P_{\mu,s}(X_t \in \cdot | \tau_X > t) = \alpha_{Bm}.$$
\end{proof}

\subsection{Quasi-ergodic distribution}
The following theorem states the existence and uniqueness of the quasi-ergodic distribution (in the sense of Definition \ref{qed}) for the process $X$. Moreover, this quasi-ergodic distribution is the probability measure $\beta_{Bm}$ defined in \eqref{quasi-ergodic}.    
\begin{theorem}
For any probability measure $\mu$ on $(-1,1)$ and any $s \geq 0$,
$$\lim_{t \to \infty} \frac{1}{t} \int_0^t \P_{\mu,s}(X_u \in \cdot | \tau_X > t)du = \beta_{Bm},$$
where $\beta_{Bm}$ is defined in \eqref{quasi-ergodic} (Proposition \ref{prop}).
\end{theorem}
\begin{proof}
As for the proof of Theorem \ref{quasi-ergodic-critical}, the following proof will be only done for $s=0$, but the statement holds for a general starting time $s$.

Let $\mu \in \cM_1((-1,1))$. We recall the notation 
$$\mu_{(s,t)} = \P_{\mu,s}(X_t \in \cdot | \tau_X > t),~~~~\forall s \leq t,$$
and, to make the reading simpler, denote
$$\mu_t := \mu_{(0,t)},~~~~\forall t \geq 0.$$
For any probability measure $\mu$ and $t \geq 0$,
\begin{align*}
&\left|\left| \int_0^1  \P_{\mu,0}(X_{qt} \in \cdot |{\tau_X>t})dq - \beta_{Bm} \right|\right|_{TV} \\&~~~~~~~~\leq \left|\left| \int_0^1  \P_{\mu,0}(X_{qt} \in \cdot|{\tau_X>t})dq -  \int_0^1  \P^Y_{\mu_{\frac{q}{2}t},\frac{q}{2}t}(Y_{qt}\in \cdot |{\tau_Y>t})dq \right|\right|_{TV} \\&~~~~~~~~~~~~~~~~~~+  \left|\left| \int_0^1  \P^Y_{\mu_{\frac{q}{2}t},\frac{q}{2}t}(Y_{qt}\in \cdot |{\tau_Y>t})dq - \beta_{Bm} \right|\right|_{TV} \\   
&~~~~~~~~\leq \left|\left| \int_0^1   \P_{\mu_{\frac{q}{2}t},\frac{q}{2}t}(X_{qt}\in \cdot |{\tau_X>t})dq -  \int_0^1  \P^Y_{\mu_{\frac{q}{2}t},\frac{q}{2}t}(Y_{qt}\in \cdot |{\tau_Y>t})dq \right|\right|_{TV} \\&~~~~~~~~~~~~~~~~~~+  \left|\left| \int_0^1  \P^Y_{\mu_{\frac{q}{2}t},\frac{q}{2}t}(Y_{qt}\in \cdot |{\tau_Y>t})dq - \beta_{Bm} \right|\right|_{TV} \\   
&~~~~~~~~\leq \int_0^1 F\left(\frac{q}{2}t\right)dq + \int_0^1 \left|\left| \P^Y_{\mu_{\frac{q}{2}t},\frac{q}{2}t}(Y_{qt}\in \cdot |{\tau_Y>t})dq - \beta_{Bm} \right|\right|_{TV} dq.
\end{align*}
By Lebesgue's dominated convergence theorem,
$$\lim_{t \to \infty} \int_0^1 F\left(\frac{q}{2}t\right)dq = 0.$$
In order to prove the convergence towards the quasi-ergodic distribution, it remains therefore to show that 
\begin{equation}
\label{asympt-prop}
\lim_{t \to \infty} \int_0^1 \left|\left| \P^Y_{\mu_{\frac{q}{2}t},\frac{q}{2}t}(Y_{qt}\in \cdot |{\tau_Y>t})dq - \beta_{Bm} \right|\right|_{TV} dq = 0.\end{equation}
The idea of the following reasoning is the same as in the critical case. Similarly, one has, for any $x \in (-1,1)$, $t \geq 0$, $q \in (0,1)$ and $f$ bounded measurable,
\begin{align*}&e^{\lambda_{Bm} \frac{(t+1)^{1-2\kappa} - (\frac{q}{2}t + 1)^{1-2\kappa}}{1-2\kappa}} \E^Y_{\mu_{\frac{q}{2}t},\frac{q}{2}t}(f(Y_{qt}) \1_{\tau_Y > t}) =  e^{\lambda_{Bm} \frac{(qt+1)^{1-2\kappa} - (\frac{q}{2}t + 1)^{1-2\kappa}}{1-2\kappa}} \E^Y_{\mu_{\frac{q}{2}t}, \frac{q}{2}t}\left(g_t\left(Y_{qt}\right) \1_{\tau_Y >qt}\right),\end{align*}
with, for any $y \in (-1,1)$
$$g_t(y) := e^{\lambda_{Bm} \frac{(t+1)^{1-2\kappa} - ({q}t + 1)^{1-2\kappa}}{1-2\kappa}} f(y) \P^Y_{y,qt}\left[\tau_{Y} > t \right].$$ 
Also define, for any $y \in (-1,1)$,
$$g_\infty(y) := f(y) \eta_{Bm}(y).$$ 
Reminding that 
$$\P^Y_{y,qt}\left(\tau_{Y} > t \right) = \P^{\tilde{B}}_{y}\left[\tau_{\tilde{B}} > \frac{(t+1)^{1-2\kappa} - ({q}t + 1)^{1-2\kappa}}{1-2\kappa} \right],$$
and using Proposition 2.3. in \cite{CV2014} applied to the process $\tilde{B}$, $(g_t)_{t \geq 0}$ converges uniformly on $(-1,1)$ towards $g_\infty$, which implies that
$$\E^Y_{\mu_{\frac{q}{2}t},\frac{q}{2}t}\left[|g_t(Y_{qt}) - g_\infty(Y_{qt})| \middle| \tau_{Y} > qt\right] \underset{t \to + \infty}{\longrightarrow} 0.$$
As a result, if one of the limit in the following equality exists, then the other limit exists also and one has
\begin{align}
\lim_{t \to \infty} e^{\lambda_{Bm} \frac{(t+1)^{1-2\kappa} - (\frac{q}{2}t + 1)^{1-2\kappa}}{1-2\kappa}} &\E^Y_{\mu_{\frac{q}{2}t},\frac{q}{2}t}(f(Y_{qt}) \1_{\tau_Y > t}) \notag \\ & = \lim_{t \to \infty} e^{\lambda_{Bm} \frac{(qt+1)^{1-2\kappa} - (\frac{q}{2}t + 1)^{1-2\kappa}}{1-2\kappa}} \E^Y_{\mu_{\frac{q}{2}t},\frac{q}{2}t}\left(g_\infty\left(Y_{qt}\right) \1_{\tau_Y >qt}\right).
\label{align2}
\end{align}
By the definition of conditional expectation, one has
\begin{align*}
\E^Y_{\mu_{\frac{q}{2}t},\frac{q}{2}t}\left(g_\infty\left(Y_{qt}\right) \1_{\tau_Y >qt}\right) & =  \E^Y_{\mu_{\frac{q}{2}t},\frac{q}{2}t}\left(g_\infty\left(Y_{qt}\right) \middle| {\tau_Y >qt}\right)  \P^Y_{\mu_{\frac{q}{2}t},\frac{q}{2}t}( {\tau_Y >qt)}.
\end{align*}
On the one hand, by \eqref{exponential-decay},
\begin{equation}
\label{prems}
\lim_{t \to \infty} \E^Y_{\mu_{\frac{q}{2}t},\frac{q}{2}t}\left(g_\infty\left(Y_{qt}\right) \middle| {\tau_Y >qt}\right) =\alpha_{Bm}(g_\infty).\end{equation}
On the other hand, using again Proposition 2.3. in \cite{CV2014} applied to the process $\tilde{B}$, we deduce that
$$\left| e^{\lambda_{Bm} \frac{(qt+1)^{1-2\kappa} - (\frac{q}{2}t + 1)^{1-2\kappa}}{1-2\kappa}} \P^Y_{\mu_{\frac{q}{2}t}, \frac{q}{2}t}( {\tau_Y >qt}) - \mu_{\frac{q}{2}t}(\eta_{Bm}) \right| \underset{t \to + \infty}{\longrightarrow} 0,$$
and, again by \eqref{exponential-decay},
$$\lim_{t \to \infty}\mu_{\frac{q}{2}t}(\eta_{Bm}) =\alpha_{Bm}(\eta_{Bm}) = 1.$$
As a result,
\begin{equation}
\label{deux}\lim_{t \to \infty} e^{\lambda_{Bm} \frac{(qt+1)^{1-2\kappa} - (\frac{q}{2}t + 1)^{1-2\kappa}}{1-2\kappa}} \P^Y_{\mu_{\frac{q}{2}t},\frac{q}{2}t}( {\tau_Y >qt})= 1.\end{equation}
Hence, we deduce from \eqref{align2},\eqref{prems} and \eqref{deux} that, for any bounded measurable function $f$, 
\begin{align*}
\lim_{t \to \infty} e^{\lambda_{Bm} \frac{(t+1)^{1-2\kappa} - (\frac{q}{2}t + 1)^{1-2\kappa}}{1-2\kappa}} &\E^Y_{\mu_{\frac{q}{2}t},\frac{q}{2}t}(f(Y_{qt}) \1_{\tau_Y > t})\\ &= 
\lim_{t \to \infty} e^{\lambda_{Bm} \frac{(qt+1)^{1-2\kappa} - (\frac{q}{2}t + 1)^{1-2\kappa}}{1-2\kappa}} \E^Y_{\mu_{\frac{q}{2}t},\frac{q}{2}t}\left(g_\infty\left(Y_{qt}\right) \1_{\tau_Y >qt}\right) \\
&= \alpha_{Bm}(g_\infty) = \int_{(-1,1)} \alpha_{Bm}(dx) f(x) \eta_{Bm}(x) = \beta_{Bm}(f).
\end{align*}
Taking $f=\1$,
$$\lim_{t \to \infty} e^{\lambda_{Bm} \frac{(t+1)^{1-2\kappa} - (\frac{q}{2}t + 1)^{1-2\kappa}}{1-2\kappa}} \P^Y_{\mu_{\frac{q}{2}t}, \frac{q}{2}t}({\tau_Y > t}) = 1.$$
Thus, for any bounded measurable function $f$,
$$\lim_{t \to \infty} \E^Y_{\mu_{\frac{q}{2}t},\frac{q}{2}t}\left(f(Y_{qt})\middle|{\tau_Y > t}\right)= \beta_{Bm}(f).$$
We conclude to \eqref{asympt-prop} by Lebesgue's theorem. 
\end{proof}

\subsection{$Q$-process}
\subsubsection{Existence of the $Q$-process}
Now, it remains to prove the existence of the $Q$-process. More precisely,
this subsection is devoted to the proof of the following theorem:
\begin{theorem}
\label{q-proc}
For any $s \leq t$ and $\mu \in \cM_1((-1,1))$, the family of probability measure $$(\P_{\mu,s}(X_{[s,t]} \in \cdot | T < \tau_X))_{T > t}$$ converges weakly when $T$ goes to infinity towards 
\begin{equation}
\label{form-qproc}
\Q_{\mu,s}(X_{[s,t]} \in \cdot) = \E_{\mu,s}\left(\1_{X_{[s,t]} \in \cdot, \tau_X > t} \frac{\eta_t(X_t)}{\mu(\eta_s)}\right),
\end{equation}
where $(\eta_t)_{t \geq 0}$ is defined in Proposition \ref{mean-ratio}.  
Moreover, for any $s \leq t$ and $\mu \in \cM_1((-1,1))$, one has
\begin{equation}
\label{aptqproc}
||\Q_{\mu,s}(X_t \in \cdot) - \Q_{\mu,s}^Y(Y_t \in \cdot)||_{TV} \leq F(s),\end{equation}
where $F$ is the same function as in Proposition \ref{asymptotic-pseudotrajectory} and $\Q^Y$ is as defined in Proposition \ref{prop}.
\end{theorem}
Before proving this theorem, let us first state the following key proposition. 
\begin{proposition}
\label{mean-ratio}
There exist a family of positive bounded functions $(\eta_s)_{s \geq 0}$ satisfying
\begin{equation}
\label{eigenfunction}
\E_{x,s}(\1_{\tau_X > t} \eta_t(X_t)) = \eta_s(x),~~~~\forall x \in (-1,1), \forall s \leq t,\end{equation}
and $H : \cM_1((-1,1)) \times \{s,t \in \R_+ : s \leq t \} \to (0,\infty)$ such that, for any $\nu \in \cM_1((-1,1))$ and $s \geq 0$, $$\lim_{t \to \infty} H(\nu,s,t) = 0,$$ and that, for any $s \leq t$ and for any $\mu, \nu \in \cM_1((-1,1))$, 
\begin{equation}
\label{unif-bound}\left|\frac{\P_{\mu,s}(\tau_X > t)}{\P_{\nu,s}(\tau_X > t)} - \frac{\mu(\eta_s)}{\nu(\eta_s)} \right| \leq H(\nu,s,t).\end{equation}
\end{proposition}
The proof of this proposition is postponed after the proof of Theorem \ref{q-proc}.
\begin{proof}[Proof of Theorem \ref{q-proc}]
Let $\mu \in \cM_1((-1,1))$ and $s \leq t$. We define $\Q_{\mu,s}(X_{[s,t]} \in \cdot)$ as in the formula \eqref{form-qproc}. Then, for any $T > t$,
\begin{align*}
||\P_{\mu,s}(X_{[s,t]} \in \cdot | &\tau_X > T) - \Q_{\mu,s}(X_{[s,t]} \in \cdot)||_{TV}\\ &= \left|\left| \E_{\mu,s}\left[ \1_{X_{[s,t]} \in \cdot, \tau_X > t} \left(\frac{\P_{X_t,t}(\tau_X > T)}{\P_{\mu,s}(\tau_X > T)} - \frac{\eta_t(X_t)}{\mu(\eta_s)}\right)\right]\right|\right|_{TV} \\
 &= \left|\left| \E_{\mu,s}\left[ \frac{\1_{X_{[s,t]} \in \cdot, \tau_X > t}}{\P_{\mu,s}(\tau_X > t)} \left(\frac{\P_{X_t,t}(\tau_X > T)}{\P_{\mu_{(s,t)},t}(\tau_X > T)} - \frac{\eta_t(X_t)}{\mu_{(s,t)}(\eta_t)}\right)\right]\right|\right|_{TV},
\end{align*}
where \eqref{eigenfunction} was used. Hence, by \eqref{unif-bound} in Proposition \ref{mean-ratio},
\begin{align*}||\P_{\mu,s}(X_{[s,t]} \in \cdot | \tau_X > T) - \Q_{\mu,s}(X_{[s,t]} \in \cdot)||_{TV} &\leq H(\mu_{(s,t)},t,T) \left|\left| \E_{\mu,s}\left[ \frac{\1_{X_{[s,t]} \in \cdot, \tau_X > t}}{\P_{\mu,s}(\tau_X > t)}\right]\right|\right|_{TV} \\
&\leq H(\mu_{(s,t)},t,T).
\end{align*} 
Since, for $s \leq t$ fixed, $\lim_{T \to \infty} H(\mu_{(s,t)},t,T) = 0$, this implies the weak convergence of the family  $(\P_{\mu,s}(X_{[s,t]} \in \cdot | T < \tau_X))_{T \geq t}$ towards $\Q_{\mu,s}$, when $T$ goes to infinity. Now, by \eqref{asymptotic-pseudotrajectory}, for any $s \leq t \leq T$ and $\mu \in \cM_1(E_s)$,
$$||\P_{\mu,s}(X_t \in \cdot | \tau_X > T) - \P^Y_{\mu,s}(Y_t \in \cdot | \tau_Y > T)||_{TV} \leq F(s).$$
Therefore, letting $T \to \infty$, one obtains : $\forall s \leq t$ and $\mu \in \cM_1(E_s)$,
$$||\Q_{\mu,s}(X_t \in \cdot) - \Q_{\mu,s}^Y(Y_t \in \cdot)||_{TV} \leq F(s).$$
\end{proof}

\subsubsection{Proof of Proposition \ref{mean-ratio}}
The remaining of the paper is dedicated to prove Proposition \ref{mean-ratio}. In the proof, two important lemmata are used. So we will start by proving these lemmata before tackling the proof of Proposition \ref{mean-ratio}.  

\begin{lemma}
\label{harnack}
\begin{enumerate}[a)]
\item For any $s \geq 0$ and $a \in (0,1)$, there exists $C_{s,a} > 0$ such that 
$$\inf_{x \in [-a,a]} \P_{x,s}(\tau_X > t) \geq C_{s,a} \sup_{x \in (-1,1)} \P_{x,s}(\tau_X > t),~~~~\forall t \geq 0.$$
\item For any $a \in (0,1)$, there exists $C_a > 0$ such that
$$\inf_{x \in [-a,a]} \P^Y_{x,s}(\tau_Y > t) \geq C_a \sup_{x \in (-1,1)} \P^Y_{x,s}(\tau_Y > t),~~~~\forall s \leq t.$$
\end{enumerate}
\end{lemma}
\begin{proof}
\begin{enumerate}[a)]
\item Let $a > 0$. To prove this, note that, for any $x \in (-1,1)$ and $t \geq s \geq 0$, 
$$\P_{x,s}(\tau_X > t) = \P^B_{(s+1)^\kappa x}\left[\tau^{(\cdot + s + 1)^\kappa}_B > t-s\right],$$
where, for any $s \geq 0$,
$$\tau^{(\cdot + s + 1)^\kappa}_B := \inf\{t \geq 0 : |B_t| = (t+s+1)^{\kappa}\}.$$
So, the Harnack inequality to show becomes: for any $t \geq 0$, 
$$\inf_{x \in [-a(s+1)^\kappa, a(s+1)^\kappa]} \P_x(\tau^{(\cdot + s + 1)^\kappa}_B > t) \geq C_{s,a} \sup_{x \in (-(s+1)^\kappa,(s+1)^\kappa)} \P_x(\tau^{(\cdot + s + 1)^\kappa}_B > t).$$
Actually, for any $t \geq 0$,
$$\inf_{x \in [-a(s+1)^\kappa, a(s+1)^\kappa]} \P_x(\tau^{(\cdot + s + 1)^\kappa}_B > t) = \P_{a(s+1)^\kappa}(\tau^{(\cdot + s + 1)^\kappa}_B > t),$$
and
$$\sup_{x \in (-(s+1)^\kappa,(s+1)^\kappa)} \P_x(\tau^{(\cdot + s + 1)^\kappa}_B > t) =  \P_0(\tau^{(\cdot + s + 1)^\kappa}_B > t).$$
Then, for any $t \geq 0$,
\begin{align*}
    \P_{a(s+1)^\kappa}(\tau^{(\cdot + s + 1)^\kappa}_B> t) &\geq \P_{a(s+1)^\kappa}\left(\tau^0_B <\tau^{(\cdot + s + 1)^\kappa}_B,\tau^{(\cdot + s + 1)^\kappa}_B > t + \tau^0_B\right) \\
&=  \E_{a(s+1)^\kappa}\left(\1_{\tau_B^0 <\tau^{(\cdot + s + 1)^\kappa}_B} \P_0(\tau^{(\cdot + s + v + 1)^\kappa}_B > t)|_{v = \tau^0_B}\right)\\
&\geq  \P_{a(s+1)^\kappa}\left(\tau^0_B <\tau^{(\cdot + s + 1)^\kappa}_B\right) \P_0\left(\tau^{(\cdot + s + 1)^\kappa}_B > t\right),
\end{align*}
where 
$$\tau_B^0 := \inf\{t \geq 0 : B_t = 0\}.$$
Then, setting $C_{s,a} :=  \P_{a(s+1)^\kappa}\left(\tau^0_B <\tau^{(\cdot + s + 1)^\kappa}_B\right)$, one has $C_{s,a} > 0$ for any $s \geq 0$ and 
  $$\inf_{x \in [-a,a]} \P_{x,s}(\tau_X > t) \geq C_{s,a} \sup_{x \in (-1,1)} \P_{x,s}(\tau_X > t),~~~~\forall t \geq 0.$$
\item This is straightforward using the Harnack inequality for a Brownian motion and using the change of time provided by the Dubin-Schwartz transformation \eqref{ds}.
\end{enumerate}
\end{proof}
Now let us state and prove Lemma \ref{apt}.
\begin{lemma}
\label{apt}
Let $a > 0$ . Then there exists a function $\chi_a : \R_+ \to \R_+$ such that, for any $s \leq t$, for any $\mu \in \cM_1((-1,1))$ and any $\nu \in \cM_1((-1,1))$ such that $\nu([-a,a]) > \frac{1}{2}$,
$$\left|\frac{\P_{\mu,s}(\tau_X > t)}{\P_{\nu,s}(\tau_X > t)} -\frac{\P^Y_{\mu,s}(\tau_Y > t)}{\P^Y_{\nu,s}(\tau_Y > t)}\right| \leq \chi_a(s),$$
with $\chi_a(s) \to 0$ when $s$ goes to infinity
\end{lemma}
\begin{proof}
Let $s \leq t$. Then, using \eqref{rev} applied to $S = (-1,1)$, for any $\mu \in \cM_1((-1,1))$,
\begin{equation}
    \label{Y-X}
    \P^Y_{\mu,s}(\tau_Y > t) = \left(\frac{s+1}{t+1}\right)^{\frac{\kappa}{2}} \E_{\mu,s}\left[\exp\left(\frac{1}{2} N'_{s,t}\right) \1_{\tau_X > t}\right],
\end{equation}
and by \eqref{cpasmal}, for any $\mu \in \cM_1((-1,1))$, 
$$1-\phi(s) \leq \frac{\E_{\mu,s}(\exp\left(\frac{1}{2} N'_{s,t}\right)\1_{\tau_X > t})}{\P_{\mu,s}(\tau_X > t)} \leq 1+\phi(s).$$
Thus, by \eqref{Y-X}, for any $\mu \in \cM_1((-1,1))$,
$$\left(\frac{s+1}{t+1}\right)^{\frac{\kappa}{2}}(1-\phi(s)) \leq \frac{\P^Y_{\mu,s}(\tau_Y > t)}{\P_{\mu,s}(\tau_X > t)} \leq \left(\frac{s+1}{t+1}\right)^{\frac{\kappa}{2}}(1+\phi(s)).$$
and, since $\phi(s) < 1$ for any $s \geq 0$, one has also,
\begin{equation}
    \label{encadrement2}
    \left(\frac{t+1}{s+1}\right)^{\frac{\kappa}{2}}\frac{1}{1+\phi(s)} \leq \frac{\P_{\mu,s}(\tau_X > t)}{\P^Y_{\mu,s}(\tau_Y > t)} \leq \left(\frac{t+1}{s+1}\right)^{\frac{\kappa}{2}}\frac{1}{1-\phi(s)}.
\end{equation}
Thus, for any $\mu, \nu \in \cM_1((-1,1))$,
\begin{equation}
    \label{encadrement}
\frac{1-\phi(s)}{1+\phi(s)} \leq \frac{\P_{\mu,s}(\tau_X > t)}{\P^Y_{\mu,s}(\tau_Y > t)}\frac{\P^Y_{\nu,s}(\tau_Y > t)}{\P_{\nu,s}(\tau_X > t)} \leq \frac{1+\phi(s)}{1-\phi(s)}.
\end{equation}
Thus, it is deduced from \eqref{encadrement} that, for any $\mu, \nu \in \cM_1((-1,1))$,
\begin{align*}
&\left|\frac{\P_{\mu,s}(\tau_X > t)}{\P_{\nu,s}(\tau_X > t)} -\frac{\P^Y_{\mu,s}(\tau_Y > t)}{\P^Y_{\nu,s}(\tau_Y > t)}\right|\\ &~~~~~~~~~~~~~~~~\leq \frac{\P^Y_{\mu,s}(\tau_Y > t)}{\P^Y_{\nu,s}(\tau_Y > t)} \left|\frac{\P_{\mu,s}(\tau_X > t)}{\P_{\nu,s}(\tau_X > t)}\frac{\P^Y_{\nu,s}(\tau_Y > t)}{\P^Y_{\mu,s}(\tau_Y > t)} -1\right| \\
&~~~~~~~~~~~~~~~~\leq  \frac{\P^Y_{\mu,s}(\tau_Y > t)}{\P^Y_{\nu,s}(\tau_Y > t)} \left[ \left(\frac{1+\phi(s)}{1-\phi(s)} -1\right) \lor \left(1-\frac{1-\phi(s)}{1+\phi(s)}\right)\right].
\end{align*}
Now, if $\nu$ is such that $\nu([-a,a]) > \frac{1}{2}$, by Lemma \ref{harnack}, for any $s \leq t$,
\begin{align*}
    \P_{\mu,s}^Y(\tau_Y > t) &\leq \sup_{x \in (-1,1)} \P_{x,s}^Y(\tau_Y > t) \\
    &\leq \frac{1}{C_a} \inf_{x \in [-a,a]} \P_{x,s}^Y(\tau_Y > t) \\
    &\leq  \frac{1}{C_a \nu([-a,a])} \P_{\nu,s}^Y(\tau_Y > t) \\
    &\leq \frac{2}{C_a}  \P_{\nu,s}^Y(\tau_Y > t).
\end{align*}
As a result, 
$$\left|\frac{\P_{\mu,s}(\tau_X > t)}{\P_{\nu,s}(\tau_X > t)} -\frac{\P^Y_{\mu,s}(\tau_Y > t)}{\P^Y_{\nu,s}(\tau_Y > t)}\right| \leq \frac{2}{C_a}  \left[ \left(\frac{1+\phi(s)}{1-\phi(s)} -1\right) \lor \left(1-\frac{1-\phi(s)}{1+\phi(s)}\right)\right].$$
It remains to set $\chi_a(t) :=  \frac{2}{C_a} \left(\frac{1+\phi(s)}{1-\phi(s)} -1\right) \lor \left(1-\frac{1-\phi(t)}{1+\phi(t)}\right)$. Then, since $\phi(s) \to 0$ when $s \to \infty$, $\chi_a$ goes also to $0$ when $s$ goes to infinity. 
\end{proof}

\begin{remark}
The inequalities \eqref{encadrement2} allows us to get that, for any $\mu \in \cM_1((-1,1))$ and $s \geq 0$,
$$\P_{\mu,s}(\tau_X > t) =_{t \to \infty} O\left(\left(\frac{t+1}{s+1}\right)^\frac{\kappa}{2} e^{-\lambda_{Bm} \frac{(t+1)^{1-2\kappa}-(s+1)^{1-2\kappa}}{1-2\kappa}}\right).$$
Hence, defining
$$\tau_{B}^{(\cdot+1)^\kappa} := \inf\{t \geq 0 : |B_t| = (t+1)^{\kappa}\},$$
then, for $\kappa < \frac{1}{2}$, for any $\mu \in \cM_1((-1,1))$,
$$\P^{B}_{\mu}(\tau_B^{(\cdot+1)^\kappa} > t) =_{t \to \infty} O\left((t+1)^\frac{\kappa}{2} e^{-\lambda_{Bm} \frac{(t+1)^{1-2\kappa}-1}{1-2\kappa}}\right).$$
This observation completes the Theorem 5 obtained by Novikov in \cite{novikov81}, which states in our case the following asymptotic order : 
$$\log \P_\mu^B(\tau_B^{(\cdot + 1)^\kappa} > t) \sim_{t \to \infty} - \lambda_{Bm} \frac{(t+1)^{1-2\kappa}-1}{1-2\kappa}.$$
\end{remark}

Now we can prove Proposition \ref{mean-ratio}.

\begin{proof}[Proof of Proposition \ref{mean-ratio}]
Let $\nu \in \cM_1((-1,1))$ and $s \geq 0$. We recall then the notation 
$$\nu_{(s,t)} := \P_{\nu,s}(X_t \in \cdot | \tau_X > t),~~~~\forall s \leq t.$$
By Theorem \ref{thm}, the family $(\nu_{(s,t)})_{s \leq t}$ converges weakly when $t$ goes to infinity towards $\alpha_{Bm}$. Thus, by Prokhorov's theorem, $(\nu_{(s,t)})_{s \leq t}$ is tight. This implies that there exists $a_s(\nu) \in (0,1)$ such that, for any $t \geq s$, $\nu_{(s,t)}([-a_s(\nu),a_s(\nu)]) > \frac{1}{2}$. \\ 
Let $\mu \in \cM_1((-1,1))$, $s \leq t$ and $T \geq 0$. Then, by the Markov property,
$$\frac{\P_{\mu,s}(\tau_X > t+T)}{\P_{\nu,s}(\tau_X > t+T)} - \frac{\P_{\mu,s}(\tau_X > t)}{\P_{\nu,s}(\tau_X > t)} = \frac{\P_{\mu,s}(\tau_X > t)}{\P_{\nu,s}(\tau_X > t)}\left(\frac{\P_{\mu_{(s,t)},t}(\tau_X > t+T)}{\P_{\nu_{(s,t)},t}(\tau_X > t+T)} - 1\right).$$
Using the same argument as in the proof of Lemma \ref{apt}, by Lemma \ref{harnack}, one has
$$\frac{\P_{\mu,s}(\tau_X > t)}{\P_{\nu,s}(\tau_X > t)} \leq \frac{2}{C_{s,a_s(\nu)}}.$$
Thus, 
$$\left|\frac{\P_{\mu,s}(\tau_X > t+T)}{\P_{\nu,s}(\tau_X > t+T)} - \frac{\P_{\mu,s}(\tau_X > t)}{\P_{\nu,s}(\tau_X > t)}\right| \leq \frac{2}{C_{s,a_s(\nu)}} \left|\frac{\P_{\mu_{(s,t)},t}(\tau_X > t+T)}{\P_{\nu_{(s,t)},t}(\tau_X > t+T)} - 1\right|.$$
Using Lemma \ref{apt}, one has
$$\left|\frac{\P_{\mu,s}(\tau_X > t+T)}{\P_{\nu,s}(\tau_X > t+T)} - \frac{\P_{\mu,s}(\tau_X > t)}{\P_{\nu,s}(\tau_X > t)}\right| \leq \frac{2}{C_{s,a_s(\nu)}}\left( \chi_{a_s(\nu)}(t) + \left|\frac{\P^Y_{\mu_{(s,t)},t}(\tau_Y > t+T)}{\P^Y_{\nu_{(s,t)},t}(\tau_Y > t+T)} - 1\right|\right).$$
Now,
\begin{align*}
 \left|\frac{\P^Y_{\mu_{(s,t)},t}(\tau_Y > t+T)}{\P^Y_{\nu_{(s,t)},t}(\tau_Y > t+T)} - 1\right| &=  \frac{|\P^Y_{\mu_{(s,t)},t}(\tau_Y > t+T) - \P^Y_{\nu_{(s,t)},t}(\tau_Y > t+T)|}{\P^Y_{\nu_{(s,t)},t}(\tau_Y > t+T)} \\
&\leq \frac{\sup_{x \in (-1,1)}\P^Y_{x,t}(\tau_Y > t+T)}{\P^Y_{\nu_{(s,t)},t}(\tau_Y > t+T)} ||\mu_{(s,t)} - \nu_{(s,t)}||_{TV} \\
&\leq \frac{4}{C_{a_s(\nu)}} \times \left(F\left(\frac{t}{2}\right) + C_{Bm} \exp\left( - \gamma_{Bm} \frac{(t+1)^{1-2\kappa} - (\frac{t}{2} + 1)^{1-2\kappa}}{1-2\kappa}\right)\right),
\end{align*}
where we used Lemma \ref{harnack} and \eqref{inequality}. \\
We conclude from all these computations that $t \to \frac{\P_{\mu,s}(\tau_X > t)}{\P_{\nu,s}(\tau_X > t)}$ is a Cauchy sequence, hence converges as $t \to \infty$. Denote by $h(s,\mu,\nu)$ the limit and set $$H(\nu,s,t) := \frac{2}{C_{s,a_s(\nu)}}\left[ \chi_{a_s(\nu)}(t) + \frac{4}{C_{a_s(\nu)}} \times \left(F\left(\frac{t}{2}\right) + C_{Bm} \exp\left( - \gamma_{Bm} \frac{(t+1)^{1-2\kappa} - (\frac{t}{2} + 1)^{1-2\kappa}}{1-2\kappa}\right)\right)\right].$$ One has therefore, for any $\mu, \nu \in \cM_1((-1,1))$, 
$$\left|\frac{\P_{\mu,s}(\tau_X > t)}{\P_{\nu,s}(\tau_X > t)} - h(s,\mu,\nu) \right| \leq H(\nu,s,t),$$
and $\lim_{t \to \infty} H(\nu,s,t) = 0$. \\
In order to complete the proof, The final steps of the proof are inspired by the proof of Proposition 3.1 in \cite{CV2016}. We define, for any $s \geq 0$,
$$\eta_s : x \to h(s,\delta_x, \delta_{0}).$$
Then, by Lebesgue's dominated convergence theorem, $$\lim_{t \to \infty} \frac{\P_{\mu,s}(\tau_X > t)}{\P_{0,s}(\tau_X > t)} = \mu(\eta_s).$$  Then, for any $\mu, \nu \in \cM_1((-1,1))$,
$$h(s,\mu,\nu) = \lim_{t \to \infty} \frac{\P_{\mu,s}(\tau_X > t)}{\P_{\nu,s}(\tau_X > t)} =  \lim_{t \to \infty} \frac{\P_{\mu,s}(\tau_X > t)/\P_{0,s}(\tau_X > t)}{\P_{\nu,s}(\tau_X > t)/\P_{0,s}(\tau_X > t)}= \frac{\mu(\eta_s)}{\nu(\eta_s)}.$$
Moreover, for any $s \leq t \leq u$,
\begin{align*}
\E_{x,s}\left(\1_{\tau_X > t} \frac{\P_{X_t,t}(\tau_X > u)}{\P_{0,t}(\tau_X > u)} \right) &= \frac{\P_{x,s}(\tau_X > u)}{\P_{0,t}(\tau_X > u)}\\
&=  \frac{\P_{x,s}(\tau_X > u)}{\P_{0,s}(\tau_X > u)} \E_{0,s}\left(\1_{\tau_X > t} \frac{\P_{X_t,t}(\tau_X > u)}{\P_{0,t}(\tau_X > u)} \right).
\end{align*}
For any $\mu \in \cM_1((-1,1))$, integrating both sides of the equation with respect to $\mu$, letting $u \to \infty$ and using Lebesgue's theorem, we deduce that, for any $s \leq t$, there exists a positive constant $c_{s,t}$ which does not depend on $\mu$ such that
$$c_{s,t} = \frac{\E_{\mu,s}(\1_{\tau_X > t} \eta_t(X_t))}{\mu(\eta_s)}.$$
In addition, for any $s \leq t \leq u$ and for any $\mu,\nu \in \cM_1((-1,1))$,
$$c_{s,t} c_{t,u} =  \frac{\E_{\mu,s}(\1_{\tau_X > t} \eta_t(X_t))}{\mu(\eta_s)}  \frac{\E_{\nu,t}(\1_{\tau_X > u} \eta_u(X_u))}{\nu(\eta_t)}.$$
Choosing $\nu = \mu_{(s,t)} = \P_{\mu,s}(X_t \in \cdot | \tau_X > t)$ and using the Markov property, we obtain
\begin{align*}
    c_{s,t}c_{t,u} &= \frac{\E_{\mu,s}(\1_{\tau_X > t} \eta_t(X_t))}{\mu(\eta_s)} \times \frac{\E_{\mu_{(s,t)},t}(\1_{\tau_X > u}\eta_u(X_u))}{\mu_{(s,t)}(\eta_t)}\\
    &= \frac{\E_{\mu,s}(\1_{\tau_X > t} \eta_t(X_t)) \times \E_{\mu,s}(\E_{X_t,t}(\1_{\tau_X > u} \eta_u(X_u))|\tau_X > t)}{\mu(\eta_s) \times \E_{\mu,s}(\eta_t(X_t) | \tau_X > t)}\\
    &= \frac{\E_{\mu,s}(\1_{\tau_X > t} \eta_t(X_t)) \times \E_{\mu,s}(\E_{X_t,t}(\1_{\tau_X > u} \eta_u(X_u))\1_{\tau_X > t})}{\mu(\eta_s) \times \E_{\mu,s}(\eta_t(X_t)  \1_{\tau_X > t})}\\
    &= \frac{\E_{\mu,s}(\1_{\tau_X > u} \eta_u(X_u))}{\mu(\eta_s)}\\& = c_{s,u}.
\end{align*} 
Because of the last equality, replacing for all $s \geq 0$ the function $\eta_s(x)$ by $\eta_s(x)/c_{0,s}$ entails \eqref{eigenfunction}. 
\end{proof}   
\textbf{Acknowledgement.} I would like to thank my Ph.D advisor Patrick Cattiaux for suggesting me to work on this interesting topic and for the attention he gave to this paper.

\bibliographystyle{abbrv}
\bibliography{biblio-william}

\end{document}